\documentclass{amsart}

\usepackage{amsfonts}
\usepackage{amsmath}
\usepackage{amssymb}
\usepackage{amsthm}
\usepackage{color}

\newtheorem{Thm}{Theorem}

\newtheorem{Prop}[Thm]{Proposition}

\theoremstyle{definition}
\newtheorem{Def}[Thm]{Definition}

\theoremstyle{remark}

\numberwithin{equation}{section}

\def\CC{{\mathbb{C}}}

\def\RR{{\mathbb{R}}}
\def\ZZ{{\mathbb{Z}}}
\def\NN{{\mathbb{N}}}

\title[Darboux theory of integrability in the sparse case]{Darboux theory of integrability in the sparse case}
\date{\today}

\author[G.~Ch\`eze]{Guillaume Ch\`eze}
\address{Institut de Math\'ematiques de Toulouse\\
Universit\'e Paul Sabatier Toulouse 3 \\
MIP B\^at 1R3\\
31 062 TOULOUSE cedex 9, France}
\email{guillaume.cheze@math.univ-toulouse.fr}

\begin{document}
	
\begin{abstract}
Darboux's theorem and Jouanolou's theorem deal with the existence of first integrals and rational first integrals of a polynomial vector field. These results are given in terms of the degree of the polynomial vector field. Here we show that we can get the same kind of results if we consider the size of a Newton polytope associated to the vector field. Furthermore, we show that in this context the bound is optimal.
\end{abstract}	
	
\maketitle

\section*{Introduction}
In this paper we study the following  polynomial differential system in $\CC^n$:
$$\dfrac{dX_1}{dt}=A_1(X_1,\ldots,X_n), \, \ldots, \, \dfrac{dX_n}{dt}=A_n(X_1,\ldots,X_n),$$
where $A_i \in \CC[X_1,\ldots,X_n]$ and $\deg A_i \leq d$.  We associate  to this polynomial differential system the polynomial derivation $D=\sum_{i=1}^nA_i(X_1,\ldots,X_n)\partial_{X_i}$. \\

The computation of  first integrals of such  polynomial differential systems is an old and classical problem. The situation is the following: we want to compute a function $\mathcal{F}$ such that the hypersurfaces $\mathcal{F}(X_1,\ldots,X_n)=c$, where $c$ are constants, give  orbits of the differential system. Thus we want to find  a function $\mathcal{F}$ such that $D(\mathcal{F})=0$.\\

In 1878, G.~Darboux \cite{Darboux} has given a strategy to find first integrals. One of the tools developed by G.~Darboux is now called \emph{Darboux polynomials}. \\
A polynomial $f$ is said to be a \emph{Darboux polynomial},  if $D(f)=g.f$, where $g$ is a polynomial. The polynomial $g$ is called the cofactor. A lot of properties of a polynomial differential system are related to Darboux polynomials of the corresponding derivation $D$, see e.g.\cite{Goriely,Dumortier_Llibre_Artes}.\\
There exists a lot of different names in the literature for Darboux polynomials, for example we can find: special integrals, eigenpolynomials, algebraic invariant hypersurfaces, or special polynomials. \\

 G.~Darboux shows in \cite{Darboux2} that \emph{if the derivation  $D$ has at least $\binom{n+d-1}{n}+1$ Darboux polynomials then $D$ has a  first integral which can be expressed by means of these polynomials}. More precisely the first integral has the following form: $\prod_i f_i^{\lambda_i}$ where $f_i$ are Darboux polynomials and $\lambda_i$ are complex numbers. This kind of integral is called today a Darboux first integral.

In 1979,   J.-P.~Jouanolou shows in his book \cite{Joua_Pfaff}, that \emph{if a derivation has at least $\binom{n+d-1}{n}+n$ Darboux polynomials then the derivation has a rational first integral}. We recall that a rational first integral is a first integral which belongs to $\CC(X_1,\ldots,X_n)$. \\
Several authors have given simplified proof for this result. M.~Singer proves this result in $\CC^2$, see \cite{Singer}. This approach is based on a work of Rosenlicht \cite{Rosenlicht}. J.-A. Weil generalizes this strategy and gives a proof for derivations in $\CC^n$, see \cite{Weil}. J. Llibre and X. Zhang gives a direct proof  of Jouanolou's result in \cite{LZ}. \\

Darboux and Jouanolou theorem are improved in \cite{LZ1,LZ2}. The authors show that we get the same kind of result if we take into account the multiplicity of Darboux polynomials. The multiplicity of a Darboux polynomial is defined and studied in \cite{Christopher_Llibre_Pereira}.\\

The Darboux theory of integrability has been successfully used in the study of some physical problems, see e.g. \cite{Valls,LlibreValls}, and in the study of  limit cycles and centers, see e.g. \cite{Christopher1994,Schlomiuk,Llibre-rodriguez}. Unfortunately, to the author knowledge, there do not exist example showing if these bounds are optimal. In this note we are going to study the situation in the sparse case. This means that we are going to consider  polynomials $A_i$ with some coefficients equals to zero. In this situation, the size of the polynomials $A_i$ is not measured by the degree but by the size of its Newton polytope. We recall that the Newton polytope of a Laurent polynomial $f(\underline{X})=\sum_{\alpha} c_{\alpha}X^{\alpha}$, where $\underline{X}=X_1,\dots,X_n$ and $\alpha$ is a multi-index $(\alpha_1,\ldots,\alpha_n) \in \ZZ^n$, is the convex hull in $\RR^n$ of the exponent $\alpha$ of all  nonzero terms of $f$. We denote this polytope by $\mathcal{N}(f)$.\\

In this note we prove a result improving Darboux and Jouanolou theorem. Our result depends on the size of a Newton polytope associated to the derivation and not on the degree $d$. Furthermore, in this context we can give example showing that the bound is optimal. This is our result:
\begin{Thm}\label{Thm}
Let $D=\sum_{i=1}^nA_i(X_1,\ldots,X_n)\partial_{X_i}$ be a derivation.
Consider generic values $(x_1,\ldots,x_n)$ in $\CC^n$ and the polytope $N_D=\mathcal{N}\Big(\sum_{i=1}^n x_i \dfrac{A_i}{X_i}\Big)$.\\
Let $B$ be the number of integer points in $N_D \cap \NN^n$, then
\begin{enumerate}
\item \label{darbouxsparse} if $D$ has at least $B+1$ Darboux polynomials then $D$ has a Darboux first integral,
\item \label{Jouanolousparse} if $D$ has at least $B+n$ Darboux polynomials then $D$ has a rational first integral. Furthermore, this bound is optimal.
\end{enumerate}

\end{Thm}

We can remark that if we consider dense polynomials $A_i$ with degree $d$, that is to say each coefficient of $A_i$ is nonzero, then $B=\binom{n+d-1}{n}$. Thus Theorem \ref{Thm} gives the classical bounds in the dense case.\\

Now, we illustrate  why these bounds can be better than the classical ones.  We give an example with $n=2$ in order to give a picture. If $A_i$ has the following form: $A_i(X_1,X_2)=c_{e,e}X_1^eX_2^e+c_{e-1,e}X_1^{e-1}X_2^e+c_{e,e-1}X_1^{e}X_2^{e-1} +c_{0,0}$, then $B=3e+2$, and $d=\deg(A_i)=2e$. In this situation we have $\binom{n+d-1}{2}=2e(2e+1)/2$. Thus for such examples Theorem \ref{Thm} gives a linear bound instead of a quadratic bound.\\

Figure \ref{fig1} shows the Newton polygon of $A_i(X_1,X_2)$, when $e=3$. The triangle corresponds to the Newton polygon of dense polynomials with total degree $6$. In this situation, Jouanolou's theorem says that if we have $23$ Darboux polynomials then we have a rational first integral. Here, our bound improves this result and means that $13$ Darboux polynomials are sufficient to construct a rational first integral.

\begin{center}
\begin{figure}[!h]
\setlength{\unitlength}{.25cm}
\begin{picture}(10,10)
\put(0.5,1){$0$}
\put(1,2){\vector(1,0){7}}
\put(1,1.5){\vector(0,1){7}}
\put(-1,8){$X_2$}
\put(8.5,1.5){$X_1$}
\put(1,2){\circle*{.2}}
\put(2,3){\circle*{.2}}
\put(3,4){\circle*{.2}}
\put(4,4){\circle*{.2}}
\put(3,5){\circle*{.2}}
\put(3,2){\circle*{.1}}
\put(2,2){\circle*{.1}}
\put(4,2){\circle*{.1}}
\put(1,3){\circle*{.1}}
\put(1,4){\circle*{.1}}
\put(1,5){\circle*{.1}}
\put(1,6){\circle*{.1}}
\put(1,7){\circle*{.1}}
\put(2,4){\circle*{.1}}
\put(2,5){\circle*{.1}}
\put(2,6){\circle*{.1}}
\put(3,2){\circle*{.1}}
\put(3,3){\circle*{.1}}
\put(4,3){\circle*{.1}}
\put(5,3){\circle*{.1}}
\put(5,2){\circle*{.1}}
\put(6,2){\circle*{.1}}
\put(7,2){\circle*{.1}}
\put(6,3){\circle*{.1}}
\put(5,4){\circle*{.1}}
\put(4,5){\circle*{.2}}
\put(3,6){\circle*{.1}}
\put(2,7){\circle*{.1}}
\put(1,8){\circle*{.1}}
\thicklines\put(1,2){\line(3,2){3}}
\thicklines\put(1,2){\line(2,3){2}}
\thicklines\put(3,5){\line(1,0){1}}
\thicklines\put(4,5){\line(0,-1){1}}
\end{picture}
\caption{Newton polygon  $\mathcal{N}\big(A_i(X_1,X_2)\big)$.}
\label{fig1}
\end{figure}
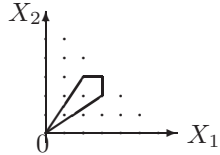
\end{center}

The Newton polygon of $\dfrac{A_i(X_1,X_2)}{X_1}$ corresponds to a translation of the Newton polygon of  $A_i(X_1,X_2)$. In Figure \ref{fig2} we show $\mathcal{N}\Big(\dfrac{A_i(X_1,X_2)}{X_1}\Big)$.

\begin{center}
\begin{figure}[!h]
\setlength{\unitlength}{.25cm}
\begin{picture}(10,10)
\put(0.5,1){$0$}
\put(0,2){\vector(1,0){8}}
\put(1,1.5){\vector(0,1){7}}
\put(-1,8){$X_2$}
\put(8.5,1.5){$X_1$}
\put(1,2){\circle*{.1}}
\put(1,3){\circle*{.2}}
\put(2,4){\circle*{.2}}
\put(4,4){\circle*{.1}}
\put(3,5){\circle*{.2}}
\put(3,2){\circle*{.1}}
\put(2,2){\circle*{.1}}
\put(4,2){\circle*{.1}}
\put(1,3){\circle*{.1}}
\put(1,4){\circle*{.1}}
\put(1,5){\circle*{.1}}
\put(1,6){\circle*{.1}}
\put(1,7){\circle*{.1}}
\put(2,4){\circle*{.1}}
\put(2,5){\circle*{.2}}
\put(2,6){\circle*{.1}}
\put(3,2){\circle*{.1}}
\put(3,4){\circle*{.2}}
\put(4,3){\circle*{.1}}
\put(5,3){\circle*{.1}}
\put(5,2){\circle*{.1}}
\put(6,2){\circle*{.1}}
\put(7,2){\circle*{.1}}
\put(6,3){\circle*{.1}}
\put(5,4){\circle*{.1}}
\put(4,5){\circle*{.1}}
\put(3,6){\circle*{.1}}
\put(2,3){\circle*{.1}}
\put(3,3){\circle*{.1}}
\put(2,7){\circle*{.1}}
\put(1,8){\circle*{.1}}
\put(0,2){\circle*{.2}}
\put(2,5){\line(1,0){1}}
\put(3,5){\line(0,-1){1}}
\thicklines\put(0,2){\line(2,3){2}}
\thicklines\put(0,2){\line(3,2){3}}
\thicklines\put(2,5){\line(1,0){1}}
\thicklines\put(3,5){\line(0,-1){1}}
\end{picture}
\caption{Newton polygon  $\mathcal{N}\Big(\dfrac{A_i(X_1,X_2)}{X_1}\Big)$.}
\label{fig2}
\end{figure}
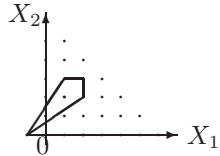
\end{center}
Figure \ref{fig3} shows the part of the Newton polygon of  $x_1\dfrac{A_1(X_1,X_2)}{X_1}+x_2\dfrac{A_2(X_1,X_2)}{X_2}$ in $\NN^2$, when $x_1$, $x_2$ are generic.
\begin{center}
\begin{figure}[!h]
\setlength{\unitlength}{.4cm}
\begin{picture}(10,10)
\put(0.5,1){$0$}
\put(1,2){\vector(1,0){7}}
\put(1,1.5){\vector(0,1){7}}
\put(-0.5,8){$X_2$}
\put(8.5,1.5){$X_1$}
\put(1,2){\circle*{.2}}
\put(2,3){\circle*{.2}}
\put(3,4){\circle*{.2}}
\put(4,4){\circle*{.2}}
\put(3,5){\circle*{.2}}
\put(3,2){\circle*{.1}}
\put(2,2){\circle*{.2}}
\put(4,2){\circle*{.1}}
\put(1,3){\circle*{.2}}
\put(1,4){\circle*{.1}}
\put(1,5){\circle*{.1}}
\put(1,6){\circle*{.1}}
\put(1,7){\circle*{.1}}
\put(2,4){\circle*{.2}}
\put(2,5){\circle*{.2}}
\put(2,6){\circle*{.1}}
\put(3,2){\circle*{.1}}
\put(3,3){\circle*{.2}}
\put(4,3){\circle*{.2}}
\put(5,3){\circle*{.1}}
\put(5,2){\circle*{.1}}
\put(6,2){\circle*{.1}}
\put(7,2){\circle*{.1}}
\put(6,3){\circle*{.1}}
\put(5,4){\circle*{.1}}
\put(4,5){\circle*{.1}}
\put(3,6){\circle*{.1}}
\put(2,7){\circle*{.1}}
\put(1,8){\circle*{.1}}
\thicklines\put(3,5){\line(1,-1){1}}
\thicklines\put(2,2){\line(2,1){2}}
\thicklines\put(1,3){\line(1,2){1}}
\thicklines\put(2,5){\line(1,0){1}}
\thicklines\put(4,4){\line(0,-1){1}}
\thicklines\put(1,2){\line(1,0){1}}
\thicklines\put(1,2){\line(0,1){1}}

\end{picture}
\caption{Newton polygon $\mathcal{N}\Big(x_1\dfrac{A_1(X_1,X_2)}{X_1}+x_2\dfrac{A_2(X_1,X_2)}{X_2}\Big)\cap \NN^2$.}
\label{fig3}
\end{figure}
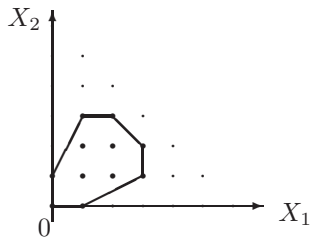
\end{center}

\subsection*{Structure of the paper} In Section \ref{section:1} we give some results about Newton  polytopes and weighted degree. In Section \ref{section:proof} we prove Theorem \ref{Thm} and we show with an example that the bound is optimal.

 \section{ToolBox}\label{section:1}
In this section we introduce some notations and results that will be useful in Section \ref{section:proof}. These kinds of tools are already present in the work of Ostrowski, see \cite{Ostro}. For more results about  sparse polynomials, see e.g. \cite{GKZ,CLO}.

\begin{Def}
Let $P$ be a polytope, then $H$ is a supporting hyperplane of $P$ if
\begin{enumerate}
\item $H \cap P \neq \emptyset$,
\item $P$ is fully contained in one of the two halfspaces defined by $H$.
\end{enumerate}
\end{Def}

In our situation, as we consider Newton polytopes, the equation of $H$ has integer coefficients.  More precisely, the equation of $H$ is $\nu. m=a$, where $.$ denotes the usual scalar product,  $\nu$ is a vector with integer coefficients, and $a$ is an integer.\\

We can represent a Newton polytope with the equations of its supporting hyperplanes:\\
$$\mathcal{N}(f)=\{ m \in \ZZ^n \mid \nu_j.m\leq a_j, \textrm{ for } j=1,\ldots,k\},$$
where $k$ is the number of supporting hyperplanes.

Now we define  a degree related to a vector $\nu \in \ZZ^n$.
\begin{Def}
Let $\nu \in \ZZ^n$, we set $\deg_{\nu}(f)=\max_{m \in \mathcal{N}(f)} \nu.m$.
\end{Def}

Now, we explain why we can call $\deg_{\nu}(f)$ a degree.

\begin{Prop}\label{prop:degree}
Let $f$ and $g$ be two polynomials in $\CC[X_1,\ldots,X_n]$, $\nu=(\nu_1,\ldots,\nu_n)$ in $\ZZ^n$ and $x_1,x_2$ two generic elements in $\CC^2$.
\begin{enumerate}
\item $\deg_{\nu}(f+g)\leq \max\big(\deg_{\nu}(f),\deg_{\nu}(g)\big)$,
\item $\deg_{\nu}(x_1f+x_2g)= \max\big(\deg_{\nu}(f),\deg_{\nu}(g)\big)$,
\item $\deg_{\nu}(f.g)=\deg_{\nu}(f)+\deg_{\nu}(g)$,
\item $\deg_{\nu}(\partial_{X_i}f) \leq \deg_{\nu}(f)-\nu_i=\deg_{\nu}(f/X_i)$.
\end{enumerate}
\end{Prop}

\begin{proof}
\begin{enumerate}
\item As $\mathcal{N}(f+g)$ is included in the convex hull of  $\mathcal{N}(f) \cup \mathcal{N}(g)$, we have 
$$\deg_{\nu}(f+g)=\max_{m \in \mathcal{N}(f+g)} \nu.m\leq \max_{m \in Conv( \mathcal{N}(f) \cup \mathcal{N}(g) )}\nu.m,$$
where $Conv(.)$ denotes the convex hull.\\
Furthemore, as we consider a convex set, we deduce that this maximum is reached on a point in $\mathcal{N}(f) \cup \mathcal{N}(g)$. Thus
\begin{eqnarray*}
\deg_{\nu}(f+g)&\leq &\max_{m \in \mathcal{N}(f) \cup \mathcal{N}(g) }\nu.m\\
&\leq&\max\big(   \max_{m \in \mathcal{N}(f)} \nu.m, \max_{m \in \mathcal{N}(g)} \nu.m \big)\\
&\leq&    \max\big(\deg_{\nu}(f),\deg_{\nu}(g)\big).
\end{eqnarray*}
\item As $x_1, x_2$ are generic then $\mathcal{N}(x_1f+x_2g)$ is equal to the convex hull of  $\mathcal{N}(f) \cup \mathcal{N}(g)$. Indeed, with generic $x_1$ and $x_2$ we avoid simplifications in the sum $x_1f+x_2g$. Then the proof in this case proceeds as before.
\item This result comes from the well-known result by Ostrowski, \cite{Ostro}, which gives: $\mathcal{N}(f.g)=\mathcal{N}(f)+\mathcal{N}(g)$, where $+$  in this situation is the Minkowski sum.
\item Let $e_i$ be the i-th vector of the canonical basis of $\RR^n$, then we have 
$$\max_{m \in \mathcal{N}(\partial_{X_i} f)} \nu.m \leq \max_{m \in \mathcal{N}(f)} \nu.(m-e_i)=\max_{m \in \mathcal{N}(f)} \nu.m -\nu_i =\deg_{\nu}(f)-\nu_i.$$
 We also have 
$$\max_{m \in \mathcal{N}(f)} \nu.(m-e_i)=\max_{m \in \mathcal{N}(f/X_i)} \nu.m$$
this completes the proof.
\end{enumerate}
\end{proof}

Newton polytopes and degree $\deg_{\nu}$ are related by the following proposition.

\begin{Prop}\label{prop:inclusion}
Let $f \in  \CC[X_1^{\pm 1},\ldots,X_n^{\pm 1}]$ be a Laurent polynomial with corresponding Newton polytope:
$$\mathcal{N}(f)=\{ m \in \ZZ^n \mid \nu_j.m\leq a_j, \textrm{ for } j=1,\ldots,k\},$$
where $\nu_j.m=a_j$, whith  $a_j \in\ZZ$, and $\nu_j \in \ZZ^n$ are the equations of the $k$  supporting hyperplanes of $\mathcal{N}(f)$.\\
Let $g \in \CC[X_1^{\pm 1},\ldots,X_n^{\pm 1}]$ such that for $j=1,\ldots,k$, $\deg_{\nu_j}(g)\leq \deg_{\nu_j}(f)$ then $\mathcal{N}(g)$ is included in $\mathcal{N}(f)$.
\end{Prop}

\begin{proof}
We just have to remark that
$$\max_{n \in \mathcal{N}(g)}  \nu_j.n=\deg_{\nu_j}(g) \leq \deg_{\nu_j}(f) =\max_{m \in \mathcal{N}(f)} \nu_j.m = a_j.$$
Thus each element in $\mathcal{N}(g)$ satisfies the equations of $\mathcal{N}(f)$.
\end{proof}
\section{Proof of Theorem \ref{Thm}}\label{section:proof}
\subsection{Newton polytope and cofactors}
We are going to show that if $A_i$ are sparse then the cofactors are sparse. This property will be the main tool of the proof of Theorem \ref{Thm}.

\begin{Prop}\label{prop:cofsparse}
Let $D=\sum_{i=1}^nA_i(X_1,\ldots,X_n)\partial_{X_i}$ be a derivation. Let $f$ be a Darboux polynomial with corresponding cofactor $g$.\\
Consider generic values $(x_1,\ldots,x_n)$ in $\CC^n$ and let $N_D$ be the convex set $\mathcal{N}\Big(\sum_{i=1}^n x_i \dfrac{A_i}{X_i}\Big)$
 then
$$\mathcal{N}(g) \subset  N_D \cap \NN^n.$$
\end{Prop}

\begin{proof}
First, obviously $\mathcal{N}(g) \in \NN^n$ since $g$ is a polynomial.\\
Second, we have $g.f=\sum_{i=1}^n A_i \partial_{X_i} f$ , thus for all $\nu \in \ZZ^n$ we have $\deg_{\nu}(g.f)=\deg_{\nu}(\sum_{i=1}^n A_i \partial_{X_i}f)$.  Thanks to Proposition \ref{prop:degree}, we deduce these inequalities
\begin{eqnarray*}
\deg_{\nu}(g)+\deg_{\nu}(f) & \leq & \max_i(\deg_{\nu}(A_i) \partial_{X_i}f)\\
& \leq & \max_i(\deg_{\nu}(A_i) +\deg_{\nu}\partial_{X_i}f)\\
& \leq & \max_i(\deg_{\nu}(A_i) +\deg_{\nu}(f) -\nu_i).
\end{eqnarray*}
Thus we have
\begin{eqnarray*}
\deg_{\nu}(g)&\leq& \max_i(\deg_{\nu} (A_i)  -\nu_i)\\
\deg_{\nu}(g)&\leq& \max_i\Big(\deg_{\nu} (\dfrac{A_i}{X_i})\Big)\\
\deg_{\nu}(g)&\leq& \deg_{\nu} \Big(\sum_{i=1}^n x_i.\dfrac{A_i}{X_i}\Big).\\
\end{eqnarray*}
Now, we apply Proposition \ref{prop:inclusion} with  $\nu\in \ZZ^n$ corresponding to supporting hyperplanes of $\mathcal{N}\Big(\sum_{i=1}^n x_i.\dfrac{A_i}{X_i}\Big)$ and we get the desired result.
\end{proof}

We can now prove easily Theorem \ref{Thm}. Indeed, we use the classical strategy to prove Darboux Theorem in our situation.\\
As for all cofactors $g_{f_i}$ associated to a Darboux polynomial $f_i$, we have, by Proposition \ref{prop:cofsparse}, 
$$\mathcal{N}(g_{f_i}) \subset  N_D \cap \NN^n,$$
then all  cofactors belong to a $\CC$-vector space of dimension  $B$, where  $B$ is  the number of integer points in $N_D \cap \NN^n$. Thus if we have $B+1$ cofactors, then there exists a relation between them: $$(\star)\,\, \sum_{i \in I} \lambda_i g_{f_i}=0,$$ 
where $\lambda_i$ are complex numbers.
Now, we recall a fundamental and  straightforward result on Darboux polynomials: $g_{f_1.f_2}=g_{f_1}+g_{f_2}$.  Thus relation $(\star)$ gives the Darboux first integral $\prod_{i \in I}f_i^{\lambda_i}$. This proves the first part of Theorem \ref{Thm}.\\

Now, in order to prove the second part of our theorem, we can use the strategy proposed in \cite{LZ}. In \cite{LZ} the authors show that if we have $n$ relations of the type $(\star)$ then we can deduce a relation with integer coefficients, i.e. $\lambda_i \in \ZZ$. This gives a first integral of this kind: $\prod_{i \in I}f_i^{\lambda_i}$ with $\lambda_i \in \ZZ$, thus this first integral belongs to $\CC(X_1,\ldots,X_n)$ and we have a rational first integral.\\
 In the sparse case as the cofactors belong to a $\CC$-vector space of dimension $B$, if we have $B+n$ Darboux polynomials then we have $B+n$ cofactors and then we deduce $n$ relations between the cofactors. With the strategy used in \cite{LZ} we obtain that the derivation has a rational first integral.
\subsection{The bound is optimal}
Consider a polynomial $p(X_1) \in \CC[X_1]$ with degree $d$. Let $\alpha$ be a root of $p$ and $\xi_2,\ldots,\xi_n$ be distinct complex numbers such that $p'(\alpha),\xi_2,\ldots,\xi_n$ are $\ZZ$-independent. We denote by $D$ the following derivation:
$$D=p(X_1)\partial_{X_1}+\xi_2X_2\partial_{X_2}+\cdots+\xi_nX_n\partial_{X_n}.$$

This derivation has no non-trivial rational first integrals, by Corollary 5.3 in \cite{Goriely}. Indeed, $(\alpha,0,\dots,0)$ is a fixed point of the polynomial vector field, and the corresponding eigenvalues are distinct and $\ZZ$ independent.\\

By Proposition \ref{prop:cofsparse}, if $g$ is a cofactor then $\mathcal{N}(g) \subset  N_D \cap \NN^n$. Here, $N_D\cap \NN^n$ is the set of univariate polynomials in $X_1$ with degree smaller than $d-1$. Thus $B=d$. \\

Furthermore, $(X_1-\alpha)$, $(X_1-\alpha_2),\ldots,(X_1-\alpha_d)$ where $\alpha_i$ are roots of $p$, and  $X_2,\ldots,X_n$ are Darboux polynomials. Thus we have $d+n-1$ Darboux polynomials.\\

In conclusion, we cannot improve the bound given in Theorem \ref{Thm}, since there exists a derivation without rational first integrals  which has  $B+n-1$ Darboux polynomials.



\end{document}